\documentclass[11pt,twoside]{article}

\usepackage{mathrsfs,amsfonts,amsmath,amssymb}
\usepackage{latexsym}
\usepackage{comment}

\newtheorem{satz}{Theorem}
\newtheorem{proposition}[satz]{Proposition}
\newtheorem{theorem}[satz]{Theorem}

\newtheorem{corollary}[satz]{Corollary}
\newtheorem{remark}[satz]{Remark}

\def\Z{\mathbb {Z}}
\def\F{\mathbb {F}}
\def\E{\mathsf{E}}

\def\a{\alpha}

\def\C{\mathbb{C}}

\def\d{\delta}
\def\o{\omega}
\def\({\big (}
\def\){\big )}

\def\G{\Gamma}

\def\dim{{\rm dim}}
\def\le{\leqslant}
\def\ge{\geqslant}
\def\_phi{\varphi}
\def\eps{\varepsilon}

\def\Gr{{\mathbf G}}
\def\FF{\widehat}

\def\ov{\overline}

\def\la{\lambda}
\def\D{\Delta}

\def\supp{\mathsf{supp}}

\def\T{\mathsf{T}}

\def\C{\mathbb{C}}

\def\SL{{\rm SL}}
\def\GL{{\rm GL}}

\def\Aff{{\rm Aff}}

\newcommand{\bp}{\bigskip}

\author{I.D. Shkredov}
\title{Some applications of representation theory to the sum--product phenomenon
}
\date{}
\begin{document}
	\maketitle


\begin{center}
	Annotation.
\end{center}

{\it \small
    In our paper, we introduce a new method for estimating incidences via  representation theory.
    We obtain several applications to various sums with multiplicative characters and to Zaremba's conjecture from number theory. 
}
\\

\section{Introduction}

Given two finite sets $A$ and $B$ of an abelian ring, define the sumset, and the product set of $A$ and $B$ as 
\begin{equation}\label{def:A+B,AB}
    A+B=\{a+b:a\in A,\, b\in B\}, \quad \quad  A\cdot B =\{ab: a\in A,\, b\in B\} \,.
\end{equation}
The sum-product phenomena was introduced by Erd\H{o}s and Szemer\'edi in paper \cite{ES} where  they proved that for an arbitrary finite subset $A$ of integers one has 
\begin{equation}\label{f:ES}
    \max\{ |A+A|, |A\cdot A| \} \gg |A|^{1+c} \,.
\end{equation}
Here  $c>0$ is an absolute constant and 
Erd\H{o}s and Szemer\'edi conjectured that any $c<1$ is admissible, at the cost of the implicit constant. 
As a general heuristic, the conjecture suggests that either $A+A$ or $AA$ is significantly larger then the original set, unless $A$ is close to a subring.
Even more generally speaking,  the sum--product phenomenon predicts that the 
an arbitrary subset of a ring cannot have good additive and multiplicative structures simultaneously.
The interested reader may consult \cite{TV} for a rather thorough treatment of sumsets and related questions, including some prior work on the sum-product problem.
The sum-product phenomenon has been extensively studied in the last few decades, the current records as of writing being
\cite{RSt} for real numbers, and \cite{MSt}  for sufficiently small sets in finite fields.

In our paper we consider the case of the ring $\Z_q := \Z/q\Z$ and 
we have deal with 
{\it large} sets $A \subseteq \Z_q$ (basically, it means that $|A|>q^{1-\kappa}$ for a certain constant $\kappa>0$). 
In the case of a prime $q$ the behaviour of the maximum from \eqref{f:ES} is fully  known thanks to the beautiful result of Garaev \cite{Garaev_large} who used some classical exponential sums bounds in his proof. Another approach was suggested 
in 
\cite{vinh2011szemeredi} and in  \cite{MP} where some finite geometry considerations were applied. 
For example, Vinh \cite{vinh2011szemeredi} proved that for an arbitrary  prime $q$ and any two sets $\mathcal{A} \subseteq \Z_q \times \Z_q$,  $\mathcal{B} \subseteq \Z_q \times \Z_q$ one has 
\begin{equation}\label{f:Vinh}
    \left| |\{ (a_1,a_2) \in \mathcal{A},\, (b_1,b_2) \in \mathcal{B}~:~ a_1 b_1 - a_2 b_2 \equiv 1 \pmod q  \}|
    - \frac{|\mathcal{A}|\mathcal{B}|}{q} \right| \le \sqrt{q |\mathcal{A}| |\mathcal{B}|} \,.
\end{equation}
In the proof he used  the fact that equation \eqref{f:Vinh} can be interpreted as a question about points/lines incidences. 
Clearly, the result above has the sum--product flavour and indeed one can use \eqref{f:Vinh} to derive some lower bounds for the maximum from \eqref{f:ES} (in the case of  large subsets of $\Z_q$, of course).

In this paper we introduce a new method of estimating sum--product quantities as in \eqref{f:Vinh} which does not use any exponential sums, as well as any considerations from the incidence geometry. It turns out that representation theory 
makes it possible 
to obtain (almost automatically) asymptotic formulae  for the number of solutions to  systems of  equations that are preserved by the actions of certain  groups.  
For example, equation \eqref{f:Vinh} can be interpreted as the equation $\det \left( {\begin{array}{cc}
	a & b \\
	c & d \\
	\end{array} } \right) =1$, where $(a_1,a_2) \in \mathcal{A}$ and $(b_1,b_2) \in \mathcal{B}$ and hence the equation respects the usual action of $\SL_2 (\Z_q)$. 
The advantage of 
our 
approach is its generality and (relative) simplicity. 
First of all, having a certain  equation, the method makes it possible to obtain  
an asymptotic formula for the number of solutions to the equation 
for composite $q$ due to the fact that representation theory for composite $q$ is usually  not so complicated  and can be reduced to the case of prime powers.
We should mention that the question about the sum--product phenomenon for general  $\Z_q$ and large sets is considered to be difficult and there are few results in this direction, see 
\cite{thang2015erdHos} and paper 
\cite{nica2017unimodular}, where the case of finite valuation rings was considered (also, see \cite{covert2012geometric}). 
Another 
statement 
of the problem concerning the sum--product results in $\Z_q$ is contained in 
\cite{fish2018product}, \cite{gyarmati2008equationsII}, 
\cite{s_Fish}.
Let us remark that 
in \cite{fish2018product} Fish 
also uses the property  of equation invariance, but combines it with classical Fourier analysis.
Secondly, due to the obvious observation  that representation theory deals  with some facts concerning the acting group but not with sets, in all our results all the  sets involved (as $\mathcal{A}, \mathcal{B}$ in \eqref{f:Vinh}) are absolutely general and do not require to have a  special structure, for example, to be Cartesian products of some other sets. 
The last constraint  is sometimes crucial for Fourier analysis manipulations,  see, e.g.,  \cite{ahmadi2007distribution}, \cite{vinh2009distribution},  although it  usually allows to obtain better error terms in 
asymptotic formulae. 


To be more specific 
let us mention  
just one result here (see Theorem \ref{t:det_inc} of Section \ref{sec:incidence} below). 
Given  positive integers $q$,$n,m$, $d=n+m$, an element $\la \in \Z_q$ and sets $\mathcal{A} \subseteq (\Z^{d}_q)^n$, $\mathcal{B} \subseteq  (\Z^{d}_q)^m$ define by $\mathcal{D}_\la (\mathcal{A}, \mathcal{B})$ the number of solutions to the equation 
\begin{equation}\label{def:D(A,B)}
    \det(a_1,\dots, a_n, b_1,\dots, b_m) \equiv \la \pmod q \,, 
    \quad \quad (a_1,\dots, a_n) \in \mathcal{A},\, \quad  (b_1,\dots, b_m) \in \mathcal{B} \,.
\end{equation} 

\begin{theorem}
    Let $q$ be an odd  prime number and $\la \neq 0$. Then 
\begin{equation}\label{f:det_inc_intr}
    \left|\mathcal{D}_\la (\mathcal{A}, \mathcal{B}) - \frac{|\mathcal{A}| |\mathcal{B}|}{q-1} \right| \ll 
q^{d^2/2 - d/4 - 3/4} \sqrt{|\mathcal{A}| |\mathcal{B}|} 
  \,. 
\end{equation}
\label{t:det_inc_intr}
\end{theorem}

In Section \ref{sec:applications} we obtain further applications of our approach to some problems of number theory. 
Our main observation is that the representation theory of $\SL_2 (\Z_q)$ 
makes it easy 
to insert {\it multiplicative} characters into all formulae with incidences and, therefore, to obtain non--trivial estimates for the  corresponding exponential sums. In the author opinion this  is a rather interesting phenomenon due to the widely--known fact that results with multiplicative characters are usually very difficult to obtain. 
As an example, we formulate the following theorem concerning summation over a hyperbolic surface. 
Denote by $\mathcal{D} \subset \C$ the unit disk.

\begin{theorem}
    Let $q$ be a prime number, $\d>0$ be a real number, $A,B, X,Y\subseteq \Z_q$ be sets, let $\chi$ be a non--principal multiplicative character and  $|X||Y| \ge q^{\d}$. 
    Also, let 
    $c_A : A \to \mathcal{D}$, $c_B : B \to \mathcal{D}$  
    be some weights. 
    Then there is $\eps(\d)>0$ such that 
\[
     \sum_{a\in A,\, b\in B,\, x\in X,\, y\in Y ~:~ (a+x)(b+y)=1} c_A (a) c_B (b)  \chi(a+x) 
    \le \sqrt{|A||B|} (|X||Y|)^{1-\eps(\d)}
     \,.
\]
\label{t:chi_hyp+}
\end{theorem}

Another application of the approach allows us to generalize \cite[Theorem 4]{sh_Kloosterman} (also, see Theorem 33 from this paper). Let $\chi$ be a non--principal multiplicative character over a finite field $\F$.
Consider the Kloosterman sum twisted by the character $\chi$, namely, 
\[
K_\chi (n,m) = \sum_{x \in \F \setminus \{0\}} \chi(x) e( nx + mx^{-1}) \,,
\]
where $e(\cdot)$ is an additive character on $\F$. 
We are interested in bilinear forms of Kloosterman sums (motivation can be found, say, in \cite{KMS_Kloosterman_gen}, \cite{sh_Kloosterman}, \cite{Shparlinski_Kloosterman})
that is, the sums of the form 
\[
S_\chi (\a,\beta) = \sum_{n,m} \a(n) \beta (m) K_\chi (n,m) \,,
\] 
where $\a : \F \to \C$, $\beta : \F \to \C$ are arbitrary  functions.

\begin{theorem}
    Let $c>0$, $\chi$ be a non--principal multiplicative character  and  $q$ be a prime number. 
	Let $t_1, t_2 \in \Z_p$, $N,M$ be integers, $N,M \le q^{1-c}$ 
	and let 
	$\a,\beta  : \Z_q \to \C$ be functions supported on $\{1,\dots, N\} +t_1$ and $\{1,\dots, M\} +t_2$, respectively.  
	Then there exists $\d = \delta (c) >0$ 
 such that  
	\begin{equation}\label{f:d-est_intr+}
	S_\chi (\a,\beta ) \lesssim 
	\| \a \|_2 \| \beta \|_2  q^{1-\d}  \,.  
	\end{equation}
 	Besides, if  
	$M^2 <pN$, then
	\begin{equation}\label{f:Kloosterman_NM_1_intr}
	S_\chi (\{1,\dots, N\}+t_1,\beta) \lesssim \| \beta \|_2 \left(N^{3/7} M^{1/7} p^{13/14} +  N^{3/4} p^{3/4} + N^{1/4} p^{13/12} \right) \,.
	\end{equation}
\label{t:chi_Kloosterman_intr}
\end{theorem}

Our bounds 
\eqref{f:d-est_intr+}, \eqref{f:Kloosterman_NM_1_intr} (also, see much more general Corollary \ref{c:chi_Kloosterman} below)
are non--trivial and do not seem to be covered by   \cite{KMS_Kloosterman_gen}, \cite{Shparlinski_Kloosterman} results, which require some additional restrictions on $\chi$.

Finally, we obtain an application to Zaremba's 
conjecture 
\cite{zaremba1972methode}. 
Recall the main  result of  \cite{MMS_Korobov}.

\begin{theorem}
    Let $q$ be a positive 
    sufficiently large 
    integer 
    with sufficiently large prime factors. 
    Then there is a positive integer $a$, $(a,q)=1$ and 
\begin{equation}\label{f:main_M}
        M= O(\log q/\log \log q)
\end{equation}
    such that 
\begin{equation}\label{f:ZM_expansion}
    \frac{a}{q} = [0;c_1,\dots,c_s] = 
    \cfrac{1}{c_1 +\cfrac{1}{c_2 +\cfrac{1}{c_3+\cdots +\cfrac{1}{c_s}}}}
    \,, \quad \quad c_j \le M\,, \quad \quad  \forall j\in [s]\,.
\end{equation}
    Also, if $q$ is a sufficiently large square--free number, then \eqref{f:main_M}, \eqref{f:ZM_expansion} take place.\\ 
    Finally, if $q=p^n$, $p$ is an arbitrary prime, then  \eqref{f:main_M}, \eqref{f:ZM_expansion}  hold for sufficiently large $n$. 
\label{t:ZM}
\end{theorem}

Using an idea from representation theory, one can generalize  Theorem \ref{t:ZM}.

\begin{theorem}
    Let $q$ be a 
    sufficiently large prime number and 
    $\Gamma \le \Z_q$ 
    be a multiplicative subgroup, 
\begin{equation}\label{c:Gamma}
    |\G| \gg \frac{q}{\log^\kappa q} \,,
\end{equation}
    where $\kappa>0$ is an absolute constant. 
    Then there is 
    $a\in \Gamma$ and 
\begin{equation}\label{c:M}
        M= O(\log q/\log \log q)
\end{equation}
    such that 
\begin{equation*}\label{f:main_expansion}
    \frac{a}{q} = [0;c_1,\dots,c_s] \,, \quad \quad c_j \le M\,, \quad \quad  \forall j\in [s]\,.
\end{equation*}
\label{t:ZM_new}
\end{theorem}

Some results of this type
concerning 
restrictions of the numerators of fractions  \eqref{f:ZM_expansion} to multiplicative subgroups were obtained in 
\cite{chang2011partial}, 
\cite{moshchevitin2007sets} and \cite{ushanov2010larcher}.

\bp 

We thank Nikolai Vavilov and Igor Shparlinski 
for  useful discussions and references. 

\section{Definitions and preliminaries}

Let $\Gr$ be  a group (commutative or not) and $A,B$ be some subsets of $\Gr$. 
The sumset (and the product set) of $A$ and $B$ was defined in \eqref{def:A+B,AB}. Let us write $A  \dotplus B$ if for finite sets $A$, $B$ one has $|A+B| = |A||B|$. 
We  use a representation function notation such as  $r_{AB} (x)$ or $r_{AB^{-1}} (x)$, which counts the number of ways $x \in \Gr$ can be expressed as the product $ab$ or  $ab^{-1}$ with $a\in A$, $b\in B$, respectively. 
For example, $|A| = r_{AA^{-1}}(1)$.
Let us write $r^{(k)}_{A}$ for $r_{A\dots A}$, where the set $A$ is taken $k$ times. 
Having real functions $f_1,\dots, f_{2k} :\Gr \to \C$
(let $k$ be an even number for concreteness), we put 
\[
	\T_{k} (f_1,\dots, f_{2k}) 
	=
	\sum_{a_1 a^{-1}_2  \dots  a_{k-1} a^{-1}_k = a_{k+1} a^{-1}_{k+2} \dots a_{2k-1} a^{-1}_{2k}} f_1(a_1) \dots f_{2k} (a_{2k}) \,.
\]
We denote the Fourier transform of a function  $f : \Gr \to \mathbb{C}$ by~$\FF{f},$ namely, 
\begin{equation}\label{F:Fourier}
\FF{f}(\xi) =  \sum_{x \in \Gr} f(x) \xi(x) \,,
\end{equation}
where 
$\xi$
is an additive character on $\Gr$. 
In this paper we use the same letter to denote a set $A\subseteq \Gr$ and  its characteristic function $A: \Gr \to \{0,1 \}$. 
Finally, if $|\Gr| < \infty$, then we consider the balanced function $f_A$ of $A$, namely, $f_A (x) := A(x) - |A|/|\Gr|$.

In this  paper we 
have deal with 
the group $\SL_2 (\Z_q) \le \GL_2 (\Z_q)$  of matrices 
\[
g=
\left( {\begin{array}{cc}
	a & b \\
	c & d \\
	\end{array} } \right) = (ab|cd) 
 = (a,b|c,d) 
 \,, \quad \quad a,b,c,d\in \Z_q \,, \quad \quad \det(g) = ad-bc=1 \,,
\] 
which acts on the project line (in the case of a prime number $q$) via the formula $gx =\frac{ax+b}{cx+d}$ and naturally acting on $\Z_q \times \Z_q$ for an arbitrary $q$. 

Now we give a simplified version of the special case of \cite[Theorem 6]{klingenberg1961orthogonal} 
(also, see 
\cite[Theorem 3]{james1973structure}). 
Let $p$ be a  prime number, $d$ be a positive integer, $V(\Z_{p^d})$ be a vector space over $\Z_{p^d}$, $\dim\, V (\Z_{p^d}) = n$ on which a non--degenerate symmetric bilinear form $\Phi (\cdot, \cdot)$ is given. 
The group of isometries of $V$ is called the {\it orthogonal group of} $V(\Z_{p^d})$, $O_n (\Z_{p^d})$ and the subgroup of isometries with determinant one is called the {\it special orthogonal group of} $V(\Z_{p^d})$, $SO_n (\Z_{p^d})$. 

\begin{theorem}
    Let $p$ be a  prime number, $p\ge 5$, $d$ be a positive integer,
    $V(\Z_{p^d})$ be a vector space over $\Z_{p^d}$, 
    and $\Phi(x_1,\dots,x_n; y_1,\dots, y_n) = x_1 y_1 + \dots + x_n y_n$ defined on $V(\Z_{p^d}) \times V(\Z_{p^d})$. 
    Suppose that $\Gamma$ is a normal subgroup of $SO_n (\Z_{p^d})$, where $n\ge 3$, $n\neq 4$. 
    Then $\Gamma$ is a congruence subgroup  with  the quotient isomorphic to $SO_n ( \Z_{p^r})$, $r< d$. 
\label{t:Klingenberg_O}
\end{theorem}

Indeed,  in \cite[Theorem 6]{klingenberg1961orthogonal} it requires to calculate the center of $SO_n (\Z_{p^d})$, which is trivial as one can easily check (or consult \cite[Lemma 1]{james1973structure} for general $\Phi$). 
Further one needs to find an isotropic vector $x=(x_1,\dots,x_n) \neq 0$ such that $\Phi(x,x) = 0$ and this is an obvious task to do as the equation $x^2_1+ x_2^2 + x_3^2 \equiv 0 \pmod p$ has a nonzero solution (and hence a solution modulo $p^d$ by Hensel's lemma).
Finally, notice  that in the case $n=4$ one can in principle use  \cite[Remark 2]{james1973structure} (in \cite[Theorem 6]{klingenberg1961orthogonal} the author considers the case $n=4$ under some additional assumptions which exclude the case of the sum of two hyperbolic planes).

Basic facts of representation theory can be found in \cite{kirillov2012elements}. Recall that a representation $\rho$ of a group $\Gr$ 
is called {\it faithful} if it is injective.
%
%
%
We need some number--theoretic functions. Given a positive integer $n$ we write $\tau (n)$ for the number of all divisors of $n$ and by $\omega (n)$ denote the number of all prime divisors. 
Also, denote by $J_k (n) = n^k \prod_{p|n} (1-p^{-k})$  the Jordan  totient function 
equals the number of 
$k$--tuples of positive integers that are less than or equal to $n$ and that together with $n$ form a coprime set of $k+1$ integers. For example, it is easy to see that $|\SL_2 (\Z_q)| = q J_2 (q)$.

The signs $\ll$ and $\gg$ are the usual Vinogradov symbols.
When the constants in the signs  depend on a parameter $M$, we write $\ll_M$ and $\gg_M$. 
If $a\ll_M b$ and $b\ll_M a$, then we write $a\sim_M b$. 
All logarithms are to base $2$.
We write $\Z_q = \Z/q\Z$ and let $\Z^*_q$ be the group of all invertible elements of $\Z_q$. 
By $\F_p$ denote $\F_p = \Z/p\Z$ for a prime $p$. 
Finally, let us denote by $[n]$ the set $\{1,2,\dots, n\}$.

\section{Applications to incidence problems}  
\label{sec:incidence}

We start with the simplest question about points/hyperplanes incidences (see equation \eqref{def:I(A,B)} below).
This problem was considered before in 
\cite{vinh2015product}, \cite{the2022dot}, where the authors obtained better asymptotic formulae for the quantity $\mathcal{I}_\la (\mathcal{A}, \mathcal{B})$ using other approaches. 
We commence with equation \eqref{def:I(A,B)} because it allows us to transparently  demonstrate our method, and because 
we will use some of the calculations from the proof  below. 
As we will see the  proof of Theorem \ref{t:inc_Z_q} exploits some facts about representation theory of $SO_n (\Z_q)$, which preserves the distance $x_1^2 + \dots +x_n^2$ in $\Z^n_q$. 
Thus, our approach is  applicable in principle  to all distance problems, for example, to the well--known  
Erd\H{o}s--Falconer distance problem  
see, e.g.,  \cite{iosevich2007erdos}.

Given  positive integers $q, n\ge 2$, an element $\la \in \Z_q$ and sets $\mathcal{A} \subseteq \Z^n_q$, $\mathcal{B} \subseteq \Z^n_q$  consisting of tuples all coprime to $q$,  
define by $\mathcal{I}_\la (\mathcal{A}, \mathcal{B})$ the number of solutions to the equation 
\begin{equation}\label{def:I(A,B)}
    a_1 b_1 + \dots +a_n b_n \equiv \la \pmod q \,. 
\end{equation}



\begin{theorem}
    Let $q, n\ge 2$ be positive integers, $\mathcal{A} \subseteq \Z^n_q$, $\mathcal{B} \subseteq \Z^n_q$ be sets and $\la \in \Z^*_q$. 
    Let $m$ be the least prime divisor of $q$ and suppose that $m\ge 5$. 
    Then 
\begin{equation}\label{f:inc_Z_q,n=2}
    \left|\mathcal{I}_\la (\mathcal{A}, \mathcal{B}) - \frac{|\mathcal{A}| |\mathcal{B}|}{q \prod_{p|q} (1-p^{-n})} \right| \le 2  
    q^{n-1} \sqrt{|\mathcal{A}| |\mathcal{B}|} \cdot (\Theta (n) m^{-n_*} )^{1/4} 
    \,,
\end{equation}
    where $n_* = 1$ for $n=2,3$ and $n_* = n-3$ for $n\ge 4$
    and further, 
    $\Theta (2) \ll \min \{ \tau (q), \log_m q\}$, 
    $\Theta (3) \ll \min \{ \log \omega (q), 1 + \omega(q)/m \}$ and $\Theta (n) \ll 1$ for $n\ge 4$. 
\label{t:inc_Z_q}
\end{theorem}
\begin{proof}
    Let $q=p^{\o_1}_1 \dots  p^{\o_t}_t$, where $m=p_1<p_2<\dots < p_t$ are primes and $\o_j$ are positive integers. 
    Also, let $a=(a_1,\dots,a_n)$, $b=(b_1,\dots,b_n)$ and let $M(a,b) = 1$ iff the pair $(a,b)$ satisfies our equation \eqref{def:I(A,B)}.  
    Considering the unitary decomposition of the hermitian matrix $M(a,b)$, we obtain 
\begin{equation}\label{f:unitary_decomposition}
    M(a,b) = \sum_{j=1}^{q^n} \mu_j u_j (a) \ov{u}_j (b) \,,
\end{equation}
    where $\mu_1 \ge \mu_2 \ge \dots $ are the eigenvalues and $u_j$ are correspondent orthonormal eigenfunctions.  
    Clearly, 
\[
    \mathcal{I}_\la (\mathcal{A}, \mathcal{B}) = \sum_{a\in \mathcal{A}, b\in \mathcal{B}} M(a,b) = \sum_{j=1}^{q^n} \mu_j \langle \mathcal{A}, u_j \rangle \ov{\langle \mathcal{B}, u_j \rangle} \,.
\]
    Let $N=J_n (q)$. 
    By the definition of the Jordan totient function the number of vectors $a = (a_1,\dots, a_n)$ such that $a_1, \dots, a_n, q$ are coprime is exactly $N$. 
    It is easy to see that $\mu_1 = q^{n-1}$ and $u_1 (x) = N^{-1/2} (1,\dots,1)\in \mathbb{R}^{N}$.
    Indeed, we fix $b$ and thanks to the Chinese remainder theorem we need to we solve linear equation  \eqref{def:I(A,B)} modulo $p^{\o_j}_j$, $j\in [t]$.
    Since $\la \in \Z^*_q$ and hence $\la \in \Z^*_{p^{\o_j}_j}$ for all $j\in [t]$, it follows that not all coefficients of \eqref{def:I(A,B)} are divided by $p_j$ and hence there are $p^{\o_j(n-1)}_j$ solutions modulo $p^{\o_j}_j$. Hence there are $q^{n-1}$ solutions in total.  
    Thus 
    we obtain 
\begin{equation}\label{f:E_bound-}
    \mathcal{I}_\la (\mathcal{A}, \mathcal{B}) - \frac{q^{n-1}|\mathcal{A}| |\mathcal{B}|}{N} 
    = 
    \sum_{j=2}^{q^n} \mu_j \langle \mathcal{A}, u_j \rangle \ov{\langle \mathcal{B}, u_j \rangle} := \mathcal{E} \,.
\end{equation}
    By the orthonormality of $u_j$ and the H\"older inequality, we get  
\begin{equation}\label{f:E_bound}
    |\mathcal{E}| \le |\mu_2| \sqrt{|\mathcal{A}| |\mathcal{B}|} \,.
\end{equation}
    Thus it remains 
    to estimate 
    the second eigenvalue $\mu_2$
    and to do this we calculate the rectangular norm of the matrix $M$, that is 
\[
    \sum_{j=1}^{q^n} |\mu_j|^4 = \sum_{a,a'} \left| \sum_{b} M(a,b) M(a',b) \right|^2 := \sigma \,,
\]
    and then $|\mu_2|$. 
    Fixing a pair  $(a,a') \in \Z^n_q \times \Z^n_q$, we need to solve the system of two linear equations 
\begin{equation}\label{f:eq_two}
    a_1 b_1 + \dots + a_n b_n \equiv \lambda \pmod q\,, \quad \quad 
    a'_1 b_1 + \dots + a'_n b_n \equiv \lambda \pmod q \,.
\end{equation} 
    It implies, in particular, that 
\begin{equation}\label{tmp:26.01_1}
    \sum_{j=2}^n b_j (a'_1 a_j - a_1 a'_j) \equiv \lambda (a_1' - a_1) \pmod q\,,
\end{equation}
    and if $a_1' - a_1 \in \Z_q^*$, say, then we obtain $q^{n-2}$ solutions by the previous argument. 
    If not, then consider all possible determinants of $2\times 2$ matrices consisting of the elements of the matrix $(1,a_1,\dots, a_n | 1,a'_1,\dots, a'_n)$. 
    Further given a tuple $(r_1,\dots, r_n)$, where $0\le r_j \le \o_j$ we consider the set $\mathcal{A}(r_1,\dots,r_n)$ of pairs 
    $(a,a') \in \Z^n_q \times \Z^n_q$
    such that $p^{r_j}_j$ is the maximal divisor of all these  determinants.
    If $(a,a') \in \mathcal{A}(r_1,\dots,r_n)$, then $a_j \equiv a'_j \pmod {p^{r_j}_j}$ and hence 
\begin{equation}\label{tmp:26.01_2}
    |\mathcal{A}(r_1,\dots,r_n)| \le \frac{q^{2n}}{\prod_{j=1}^t p^{r_j n}_j}  \,.
\end{equation}
    To solve \eqref{tmp:26.01_1} (recall that we consider the case when $(a,a') \in \mathcal{A}(r_1,\dots,r_n)$) 
    one can use 
    the Chinese remainder theorem again and we see that there are 
    $$
        \prod_{j=1}^t p^{\o_j (n-2) + r_j}_j = q^{n-2} \prod_{j=1}^t p^{r_j}_j 
    $$
    solutions to equation \eqref{tmp:26.01_1}. 
    Combining the last bound with \eqref{tmp:26.01_2}, we obtain 
\[
    \sigma \le q^{4n-4} \sum_{r_1 \le \o_1,\dots,r_t \le \o_t}\,  \prod_{j=1}^t p_j^{-r_j(n-2)}
    = 
    q^{4n-4} \Theta(n) \,,
\]
    where $\Theta (n) = O(1)$ for $n\ge 4$, $\Theta (3) = O(\log t)$ and $\Theta (2) \ll \prod_{j=1}^t (1+\o_j) = \tau (q)$. 
    Let us remark other bounds for $\Theta (2)$ and for $\Theta (3)$, namely, from $m^{\tau (q)} \le q$ one has $\Theta (2) \ll \tau (q) \ll \log_m q$ and, clearly, $\Theta (3) \ll 1 + t/m$.

    It is instructive to consider the case $n=2$ separately. 
    Redefining the set $\mathcal{A}$, we need to solve the equation 
\begin{equation}\label{eq:n=2}
    a_1 b_1 - a_2 b_2 \equiv \la \pmod q \,, \quad (a_1,a_2) \in \mathcal{A},\, \quad (b_1,b_2) \in \mathcal{B} \,,
\end{equation}
    where $a=(a_1,a_2)$ and $b=(b_1,b_2)$.
    It is clear that our equation \eqref{eq:n=2} has the form $\det(a|b) \equiv \la \pmod q$ and hence we enjoy the following invariance property 
\begin{equation}\label{f:invariance}
    M(a,b) = M(ga, gb) \,, \quad \quad \forall g\in \SL_2 (\Z_q) \,.
\end{equation}
    Hence if $f$ is an eigenfunction of $M$ with the eigenvalue $\mu$, then for $f^g (x) := f(gx)$ one has 
\[
    \sum_{a,b} M(a,b) f^g(b) = \sum_{a,b} M(a,b) f(gb) =
    \sum_{a,b} M(ga,gb) f(gb) = \mu f(ga) = \mu f^g (a) \,,
\]
    where we have used \eqref{f:invariance} and the  transitivity of the natural action of $\SL_2 (\Z_q)$. 
    In other words, $\SL_2 (\Z_q)$ preserves the eigenspace $L_\mu$, which 
    corresponds 
    to $\mu$. 
    Now consider an arbitrary eigenfunction $u_j$, $j>1$. 
    We know that $\sum_x u_j (x) =0$ and hence $u_j$ is not a constant function. 
    There are many ways 
    to see that $\dim (L_{\mu_j})>1$ or, in other words, 
    that $\langle \{ u^g_j \}_{g\in \SL_2 (\Z_q)} \rangle \neq \langle u_j \rangle = L_{\mu_j}$.
    For example, 
    one can use the transitivity again. 
    Another approach is to notice that  the group $\SL_2 (\Z_q)$ has no non--trivial one--dimensional representations but, on the other hand, any one--dimensional eigenspace would give us a character 
(the same holds in the general case which  will be considered below).
    Thus anyway  we conclude 
    that for any $j>1$ the multiplicity of each $\mu_j$ is at least the minimal dimension of  non--trivial representations of $\SL_2 (\Z_q)$.

    Now  we essentially  repeat the argument from \cite[Theorem 12]{s_Fish}. \
    Another way is to use the first part of \cite[Theorem 1]{bardestani2015quasi} which says exactly the same. 
    So, let us 
    repeat what is known about 
    representation theory  of the group $\SL_2 (\Z_q)$, see  \cite[Sections 7, 8]{bourgain2008expansion}. 
    First of all, for any irreducible representation $\rho_q$ of $\SL_2 (\Z_q)$ we have $\rho = \rho_q = \rho_{p^{\rho_1}_1} \otimes \dots \otimes \rho_{p^{\rho_t}_t}$ and hence it is sufficient to understand  representation theory for $\SL_2 (\Z_{p^d})$, where $p$ is a prime number and $d$ is  a positive integer.  
    Now by \cite[Lemma 7.1]{bourgain2008expansion} we know that for any odd prime the dimension of any faithful irreducible representation of $\SL_2 (\Z_{p^d})$ is at least $2^{-1} p^{d-2} (p-1)(p+1)$ (a similar proof for $\SL_n (\Z_{p^d})$, $n\ge 2$ can be found in \cite[Theorem 1]{bardestani2015quasi}). 
    If $d=1$, then the classical result of Frobenius \cite{frobenius1896gruppencharaktere} says that the minimal dimension of any non--trivial representation is at least $(p-1)/2$. 
    For an arbitrary positive integer $r\le d$ we can consider the natural projection $\pi_r : \SL_2 (\Z_{p^d}) \to \SL_2 (\Z_{p^r})$ and let $H_r = \mathrm{Ker\,} \pi_r$. One can show that the set $\{ H_r \}_{r\le d}$ gives all normal subgroups of $\SL_2 (\Z_{p^d})$ and hence any nonfaithful irreducible representation arises as a faithful irreducible representation of $\SL_2 (\Z_{p^r})$ for a certain $r<d$.  Anyway, we see that  the multiplicity (dimension) $d_\rho$ of any non--trivial irreducible representation $\rho$ of  $\SL_2 (\Z_{p^d})$ is at least $(m-1)/2 \ge m/3$.

    Returning to the quantity $\sigma$, 
    we get 
    (below  $n=2$) 
\[
    |\mu_2|^4 m \le 3 \sigma \le 3 q^{4n-4} \Theta(n)
\]
    and hence 
\begin{equation}\label{f:mu2}
|\mu_2| \le \left( 3 m^{-1} q^{4n-4} \Theta(n) \right)^{1/4} 
=
    q^{n-1} (3 \Theta(n) m^{-1} )^{1/4}
\end{equation} 
    Recalling \eqref{f:E_bound-}, \eqref{f:E_bound}, we obtain the required result for $n=2$.

    Now let $n>2$. It remains only to find  a good  lower bound for the multiplicity of $\mu_j$, $j>1$ (the fact $\dim (L_{\mu_j}) > 1$ is immediate consequence that $SO_n (\Z_q)$ has no  non--trivial one--dimensional representations or thus see paper  \cite{nica2017unimodular}).
    In the higher--dimensional case $n>2$ our form $\Phi(a,b) = a_1 b_1 + \dots + a_n b_n$ is preserved by the group of orthogonal transformations $O_n (\Z_q)$ (as well as $SO_n (\Z_q)$) and hence our task is to find a good  lower bound 
    for the dimension of any non--trivial irreducible representation  of  $SO_n (\Z_{p^d})$. 
    Using Theorem \ref{t:Klingenberg_O} and the arguments as above, we see that it is enough to have deal with faithful representations and this problem was solved in \cite{bardestani2017faithful}. 
    The authors prove that the minimal dimension of any faithful representations coincides (up to constants) with the classical lower bound for minimal dimension of an arbitrary non--trivial representation for 
    split Chevalley groups over $\F_{p^d}$, see  \cite{landazuri1974minimal}, \cite{seitz1993minimal}. 
    These results combining  with the existence of isomorphisms between low--dimensional classical groups (see \cite[Proposition 2.9.1]{KL_subgroups_book}, for example) give us 
    $d_\rho \ge  2^{-2} p^{n-3}$ for $n\ge 4$ and $d_\rho \ge  2^{-2} p$ for $n=3$. 
    For $n=4$ one cannot apply Theorem \ref{t:Klingenberg_O} but it is easy to see that in this case the multiplicity of $\mu_2$ is at least $d_\rho \ge  2^{-1} (p^{}-1)$ due to the fact that the group $\SL_2 (\Z_{p^d}) \times \SL_2 (\Z_{p^d})$ acts on the quadruples $(a_1,\dots,a_4)$ and we can use previous arguments concerning $\SL_2 (\Z_{p^d})$ and the case $n=2$. 
    It follows that 
    for any $n\ge 2$ the multiplicity of $\mu_2$ is at least $\Omega(m^{-n_*})$.
%
%
    This completes the proof. 
$\hfill\Box$
\end{proof}

\bp 

Thus, as the reader can see, our method almost automatically gives some asymptotic formulae  for the number of solutions to  systems of  equations that are preserved by the actions of certain  groups.  
The only thing we need to calculate is the first eigenfunction  of the correspondent operator and its rectangular norm. 
After that we use quasi--random technique in the spirit of 
papers     \cite{Gamburd_PhD}, \cite{Huxley_quasi} and \cite{SX}.

\bigskip

Now we are ready to obtain Theorem \ref{t:det_inc_intr} from the introduction and for simplicity we consider the case of a prime number $q$. We assume that the sets $\mathcal{A}$, $\mathcal{B}$ consisting of linearly independent tuples  because otherwise there is no solutions to equation \eqref{def:D(A,B)}. 

\begin{theorem}
    Let $q$ be an odd  prime number and $\la \neq 0$. Then 
\begin{equation}\label{f:det_inc}
    2^{-3} \left|\mathcal{D}_\la (\mathcal{A}, \mathcal{B}) - \frac{|\mathcal{A}| |\mathcal{B}|}{q} \right| \le 
q^{d^2/2 - d/4 - 3/4} \sqrt{|\mathcal{A}| |\mathcal{B}|} 
    + 
    \frac{|\mathcal{A}| |\mathcal{B}|}{q^{2}}
  \,. 
\end{equation}
\label{t:det_inc}
\end{theorem}
\begin{proof} 
    The case $n=m=1$ was considered in Theorem \ref{t:inc_Z_q}, so we assume that $\max \{ n,m\} \ge 2$. 
    Let 
    $a=(a_1,\dots,a_n)$, $b=(b_1,\dots,b_m)$ and let $M(a,b) = 1$ iff the pair $(a,b)$ satisfies our equation \eqref{def:D(A,B)}.  
    Considering the singular decomposition of the  matrix $M(a,b)$, we obtain 
\[
    M(a,b) = \sum_{j=1}^{q^d} \la_j u_j (a) \ov{v}_j (b) \,,
\]
    where $\la_1 \ge \la_2 \ge \dots \ge 0$ are the singularvalues and $u_j$, $v_j$ are correspondent orthonormal singularfunctions.  
    Let 
    $$\mathcal{N}=(q^d-q^m) (q^d- q^{m+1}) \dots (q^d - q^{d-1}) 
    = q^{dn} \prod_{j=1}^{n} (1-q^{-j})
    \quad 
    \mbox{and}
    \quad 
    \mathcal{M} =q^{dm} \prod_{j=1}^m (1-q^{-j}) \,.
    $$
    It is easy to calculate $\la_1$ and to show  that $u_1 (a) = \mathcal{N}^{-1/2} (1,\dots,1)\in \mathbb{R}^{\mathcal{N}}$, as well as $v_1 (b) = \mathcal{M}^{-1/2} (1,\dots,1)\in \mathbb{R}^{\mathcal{M}}$.
    Indeed, for any fixed $a$ or $b$ we need to solve the equation  
    $\det(a|b) = \la$ in $b$ or $a$, correspondingly.
    It is easy to see that the equation $\det (a|b) = \la$, $a$ is fixed,  has $q^{dm-1} \prod_{j=2}^{m} \left( 1-q^{-j} \right) = \frac{\mathcal{M}}{q-1}$ solutions due to the number of independent vectors over $\Z_q$. 
    Similarly, the second equation  has 
    $q^{dn-1} \prod_{j=2}^{n} \left( 1-q^{-j} \right) = \frac{\mathcal{N}}{q-1}$ solutions in $a$. 
    Thus, these numbers do not depend on $a$ and $b$ and hence, indeed we have $u_1 (a) = \mathcal{N}^{-1/2} (1,\dots,1)\in \mathbb{R}^{\mathcal{N}}$, $v_1 (b) = \mathcal{M}^{-1/2} (1,\dots,1)\in \mathbb{R}^{\mathcal{M}}$ and 
    $$
        \la_1 = \langle M u_1, v_1 \rangle = \frac{\mathcal{M}}{q-1} \cdot \mathcal{N} \cdot (\mathcal{M}\mathcal{N})^{-1/2}
        =   \frac{\sqrt{\mathcal{M}\mathcal{N}}}{q-1} \,.
    $$
    Thus we get
\begin{equation}\label{f:D_error}
    \left| \mathcal{D}_\la (\mathcal{A}, \mathcal{B}) - \frac{|\mathcal{A}| |\mathcal{B}|}{q-1}\right| 
    = 
    \left| \sum_{j=2}^{q^d} \la_j \langle \mathcal{A}, u_j \rangle \ov{\langle \mathcal{B}, v_j \rangle} \right| \le \la_2 \sqrt{|\mathcal{A}| |\mathcal{B}|} \,.
\end{equation}
    As above we need to estimate the rectangular norm of the matrix $M$ that is 
\[
    \sum_{j=1}^{q^d} \la_j^4 = \sum_{a,a'} \left| \sum_{b} M(a,b) M(a',b) \right|^2 \,,
\]
    and thus 
    we arrive to the system of equations 
    $\det(a'|b) = \det (a|b) = \la$ with fixed $a$ and $a'$. 
    Fixing vectors $b_1,\dots, b_{m-1}$ we 
    have exactly 
    equation \eqref{f:eq_two} which has at most $q^{md-2}$ solutions. 
    Thus 
\begin{equation}\label{tmp:23.02_1}
    \sum_{j=1}^{q^d} \la_j^4 \le q^{2nd} q^{2md-4} = q^{2d^2-4} \,.
\end{equation} 
    Now it is easy to see that 
\[
    M(ga, gb) = M(ga_1, \dots, ga_n, g b_1,\dots, g b_m) = M(a,b)
\]
    for an arbitrary $g\in \SL_d (\Z_q)$ and thus any $\la_j$, $j>1$ has multiplicity equals the minimal  dimension
of any non--trivial irreducible representation of $\SL_d(\Z_q)$.
Thus the multiplicity of $\la_2$ is at most $2^{-2} q^{d-1}$ and hence 
$$
    \la_2 \le 2 q^{d^2/2 -1} q^{-(d-1)/4} = 
    2 q^{d^2/2 - d/4 - 3/4} \,.
$$
Using the last estimate, 
and returning to formula \eqref{f:D_error}, we obtain 
\[
    2^{-3} \left| \mathcal{D}_\la (\mathcal{A}, \mathcal{B}) - \frac{|\mathcal{A}| |\mathcal{B}|}{q} \right|
    \le q^{d^2/2 - d/4 - 3/4} \sqrt{|\mathcal{A}| |\mathcal{B}|} 
    + 
    \frac{|\mathcal{A}| |\mathcal{B}|}{q^{2}}
\]
    as required. 
$\hfill\Box$
\end{proof}

\bp 

Finally, we consider an example with the cross--ratio $[a,b,c,d]:= \frac{(a-c)(b-d)}{(a-d)(b-c)}$. 
As one can see, representation theory almost immediately gives asymptotic formula \eqref{f:cross-ratio_t} with an acceptable error term.  
Let $q$ be a prime number, $\la \in \Z_q$ and $\mathcal{A} \subseteq \Z_q \times \Z_q$,  $\mathcal{B} \subseteq \Z_q \times \Z_q$ be sets.
Define 
\begin{equation}\label{f:cross-ratio}
     \mathcal{C}_\la (\mathcal{A}, \mathcal{B}) := |\{ (a_1,a_2) \in \mathcal{A},\, (b_1,b_2) \in \mathcal{B}~:~ [a_1,a_2,b_1,b_2] \equiv \la \pmod q  \}| \,.
\end{equation}

\begin{theorem}
    Let $q$ be a prime number, $\la \in \Z_q$, $\la \neq 0,1$ and $\mathcal{A} \subseteq \Z_q \times \Z_q$,  $\mathcal{B} \subseteq \Z_q \times \Z_q$ be sets. Then  
\begin{equation}\label{f:cross-ratio_t}
 \left|\mathcal{C}_\la (\mathcal{A}, \mathcal{B}) - \frac{|\mathcal{A}| |\mathcal{B}|}{q} \right| \le 4 q^{3/4} \sqrt{|\mathcal{A}| |\mathcal{B}|} \,.
\end{equation} 
\label{t:cross-ratio}
\end{theorem}
\begin{proof} 
    As usual let $a=(a_1,a_2)$, $b=(b_1,b_2)$ and let $M(a,b)=1$
    iff the pair $(a,b)$ satisfies our equation \eqref{f:cross-ratio}. 
    It is well--known that  $\SL_2 (\Z_q)$ preserves the cross--ratio in the sense 
\begin{equation}\label{f:M_cross-ratio}
    M(ga, gb) = M(ga_1, ga_2, g b_1, g b_2) = M(a,b) \,.
\end{equation}    
    Considering the unitary decomposition of the hermitian matrix $M(a,b)$ as in \eqref{f:unitary_decomposition} we see that the property $u_1 (a) =  q^{-1} (1,\dots,1) \in \mathbb{R}^{q^{2}}$ 
    automatically follows from \eqref{f:M_cross-ratio} and 2--transitivity of $\SL_2 (\Z_q)$ on the projective line. 
    It remains to calculate the rectangular norm of the matrix $M$, that is to solve the system $[x,y,c,d]=[x,y,c',d']=\la$. 
    It follows that 
\begin{equation}\label{tmp:13.06_1}
    xy(1-\la) + (\la c-d)x + (\la d-c) y + dc(1-\la) = 0\,, 
\end{equation}
    and 
\begin{equation}\label{tmp:13.06_2}
    xy(1-\la) + (\la c'-d')x + (\la d'-c') y + d'c'(1-\la) = 0 \,.
\end{equation}
    Subtracting \eqref{tmp:13.06_2} from \eqref{tmp:13.06_1} we arrive to the equation 
\begin{equation}\label{tmp:13.06_3} 
    x (\la (c-c') + d'-d) + y(\la (d-d') + c'-c) + (1-\la) (dc-d'c') = 0
\end{equation}
    and this is a non--trivial equation excluding two cases: $c=c'$, $d=d'$ and $\la = -1$, $c=d'$, $d=c'$. 
    If equation \eqref{tmp:13.06_3} is non--trivial, then we substitute, say, $x$ into \eqref{tmp:13.06_1} and obtain at most $4$ solutions in $x,y$ (one can check that we obtain a non--trivial equation thanks to our condition $\la \neq 0,1$). 
    In the exceptional cases we have just one equation, say, \eqref{tmp:13.06_1}, and it is easy to see that our equation has at most $2q$ solutions.
    Thus 
\[
    \sum_{j=1}^{q^2} \mu_j^4 = \sum_{a,a'} \left| \sum_{b} M(a,b) M(a',b) \right|^2 
    \le 16 q^4 + 2 q^2 (2q)^2 = 24 q^4 
    \,.
\]
It remains to use  the Frobenius Theorem \cite{frobenius1896gruppencharaktere} about minimal representations of $\SL_2 (\Z_q)$.
This result gives us the bound 
$\mu_2 \le 4 q^{3/4}$ and we can apply the arguments as in the proofs of Theorems \ref{t:inc_Z_q}, \ref{t:det_inc}.    
    This completes the proof. 
$\hfill\Box$
\end{proof}

\section{On sums with multiplicative characters over some manifolds and other applications}
\label{sec:applications}

In this section we want to extend representation theory methods to some sums with multiplicative characters.  
Below $p$ is a prime number and $\F$ is a finite field of characteristic $p$.  
Let us consider 
a basic 
example. 
We know that $\SL_2 (\F)$ acts on  the projective line and it gives us an irreducible representation of this group but from \cite{kirillov2012elements}, say, it is well--known that there are other irreducible  representations of  $\SL_2 (\F)$ and a half of them are connected with ``projective lines'' equipped with multiplicative characters $\chi$. 
More precisely, it means that we consider the family  of functions $f : \F \times \F \to \C$ such that 
\begin{equation}\label{f:f_chi}
    f(\la x, \la y) = \chi (\la) f(x,y)\,, \quad \quad \forall \la \in \F^* 
    \quad \mbox{ and } \quad  \forall (x,y)\in (\F \times \F) \setminus \{0\} \,,
\end{equation} 
and now  $\SL_2 (\F)$ acts on this  family, as well as on $\F \times \F$ in a natural way. 
In our results below we do not need to use the knowledge of concrete irreducible representations of $\SL_2 (\F)$ (and other groups) but we will use only definition \eqref{f:f_chi} somehow.

Let us start with 
the following 
auxiliary proposition 
concerning summation over a hyperbolic surface (twisted by a multiplicative character) in the spirit of paper 
\cite{skorobogatov1992exponential}, say.

\begin{proposition}
    Let $A,B\subseteq \F_p$ and $G\subseteq \GL_2 (\F_p)$ be sets
    and $\chi$ be a non--trivial multiplicative character.
    Also, let 
    $c_A : A \to \mathcal{D}$, $c_B : B \to \mathcal{D}$  
    be some weights. 
    Then for any integer $k\ge 2$ the following holds 
\[
    2^{-2} \left| \sum_{a,b} c_A (a) c_B (b) \sum_{g\in G ~:~ ga =b} \chi(\gamma a+\d) \right| 
\]
\begin{equation}\label{f:chi_inc}
    \le 
    \sqrt{|A||B||G|} \cdot \T^{1/8k}_{2k} (f_G) 
    + \sqrt{|A||B|} |G| \cdot (\max \{|A|, |B|\})^{-1/2k} \,.
\end{equation}
\label{p:chi_inc}
\end{proposition}
\begin{proof} 
    Consider the functions $\mathcal{A} (\la a,\la) = c_A (a) \overline{\chi (\la)} = \mathcal{A} (\ov{x})$, $\mathcal{B} (\mu b,\mu) = c_B (b) \chi(\mu) = \mathcal{B} (\ov{y})$, where $a\in A$, $b\in B$, $\ov{x} = (x_1,x_2)$, $\ov{y} = (y_1,y_2)$ and $\mu,\la$ run over $\F^*_p$. 
    It is easy to see that we always have $\sum_{a,\la} \mathcal{A} (\la a,\la) =0$, as well as   $\sum_{b,\mu} \mathcal{B} (\mu b,\mu) = 0$ 
    since 
    $\chi$ is a non--trivial character. 
    Notice that 
\begin{equation}\label{tmp:AB}
    \sigma:= \sum_{\ov{x},\ov{y}} \mathcal{A} (\ov{x}) \mathcal{B} (\ov{y}) \sum_{g\in G ~:~ g \ov{x} = \ov{y}} 1 =
    (p-1)\sum_{a,b} c_A (a) c_B (b) \sum_{g\in G ~:~ ga =b} \chi(\gamma a+\d) 
\end{equation} 
    for any trivial/non--trivial multiplicative character $\chi$. 
    We can interpret  the left--hand side of \eqref{tmp:AB} as 
    the number of some 
    points 
    on a hyperbolic surface  counting  with weights $\mathcal{A} (\ov{x})$, $\mathcal{B} (\ov{y})$. 
    The H\"older inequality (see \cite[Lemma 13]{sh_Kloosterman})  gives us 
\[
    \sigma^{2k} \le \|\mathcal{A}\|^{2k}_2
     \|\mathcal{B}\|^{2k-2}_2 
     \sum_{h} f^{(k)}_{G G^{-1}} (h) \sum_{x} \mathcal{B}(x) \mathcal{B}(h x) \,.
\]
    Applying identity \eqref{tmp:AB}, it is easy to see that the contribution of the terms with $\sum_{x} \mathcal{B}(x) \mathcal{B}(h x) \le 32 p$, say,  corresponds to the second term from \eqref{f:chi_inc}. 
    Now using 
    \cite[Lemma 12]{sh_Kloosterman} 
    (we notice that, say, $4$ different points uniquely determine the transformation from $\GL_2 (\F_p)$), 
    combining with 
    the H\"older inequality again, we derive 
\[
    \sigma^{2k} \le  
    (|A| (p-1))^k  (|B| (p-1))^{k-1}  
\]
\[
\times 
    \left( \sum_{h} (f^{(k)}_{G G^{-1}} (h))^2 \right)^{1/4}
    \left( \sum_{h} |f^{(k)}_{G G^{-1}} (h)| \right)^{1/2}
    \left( \sum_{h} \left( \sum_{x} \mathcal{B}(x) \mathcal{B}(h x) \right)^4 \right)^{1/4}
\]
\begin{equation}\label{tmp:29.04_1}
    \le 4^k (|A| (p-1))^k  (|B| (p-1))^{k-1} \T^{1/4}_{2k} (f_G) |G|^k \cdot |B| (p-1) \,.
\end{equation} 
    Recalling \eqref{tmp:AB}, we see that estimate  \eqref{tmp:29.04_1} is equivalent to the required bound  \eqref{f:chi_inc}.  
    This completes the proof. 
$\hfill\Box$
\end{proof}

\bigskip

Now we obtain some concrete applications of Proposition \ref{p:chi_inc}, which correspond to Theorems \ref{t:chi_hyp+}, \ref{t:chi_Kloosterman_intr} of the introduction. 
Let 
$A,B, X,Y\subseteq \F_p$ be sets.
Consider the equation 
\begin{equation}\label{eq:hyp}
    (a+x)(b+y) \equiv 1 \pmod p 
\end{equation}
or, in other words, 
$y=-b+ 1/(a+x) = g_{a,b} x$, where $\det (g_{a,b}) = -1$. 
The energy $\T_{2k} (f_G)$ of the correspondent family of transformations $G = \{g_{a,b}\}_{a\in A,b\in B}$ was estimated many times see, e.g., paper 
\cite{sh_Kloosterman}. 
Applying Proposition \ref{p:chi_inc} to this particular case of equation \eqref{eq:hyp}, we obtain

\begin{corollary}
    Let $\d>0$ be a real number, $A,B, X,Y\subseteq \F_p$ be sets, let $\chi$ be a non--principal multiplicative character and $|X||Y| \ge p^{\d}$. 
    Also, let 
    $c_A : A \to \mathcal{D}$, $c_B : B \to \mathcal{D}$  
    be some weights. 
    Then there is $\eps(\d)>0$ such that 
\[
    \sum_{a,b,x,y ~:~ (a+x)(b+y)=1} c_A (a) c_B (b)  X(x) Y(y) \chi(a+x) 
    \le \sqrt{|A||B|} (|X||Y|)^{1-\eps(\d)}
     \,.
\]
\label{c:chi_hyp}
\end{corollary}

The above corollary immediately implies Theorem \ref{t:chi_Kloosterman_intr} from the introduction (compare with \cite[Theorems 4, 33, 34]{sh_Kloosterman}) which we recall here for the 
reader's convenience. 
The proof repeats the argument of \cite[Theorems 33, 34]{sh_Kloosterman}, the only small difference is the absence of the third term in formula (72) of \cite[Theorems 34]{sh_Kloosterman} due to the fact that our $\chi$ is non--principal. 
Below the sign $\lesssim$  means a multiple of the form $\log^{O(1)} (MN \| \FF{\a} \|_\infty \| \FF{\beta} \|_\infty)$.

\begin{corollary}
    Let $c>0$,  $\chi$ be a non--principal multiplicative character and $p$ be a prime number. 
	Let $t_1, t_2 \in \F_p$, $N,M$ be integers, $N,M \le p^{1-c}$ 
	and let 
	$\a,\beta  : \F_p \to \C$ be functions supported on $\{1,\dots, N\} +t_1$ and $\{1,\dots, M\} +t_2$, respectively.  
	Then there exists $\d = \delta (c) >0$ 
 such that  
	\begin{equation}\label{f:d-est_intr}
	S_\chi (\a,\beta ) \lesssim 
	\| \a \|_2 \| \beta \|_2  p^{1-\d}  \,.  
	\end{equation}
	Besides, 
	\begin{equation}\label{f:Kloosterman_NM_1}
		S(\a,\beta) \lesssim \| \beta \|_2 \left(\|\FF{\a}\|_{L^{4/3}} N^{7/48} M^{7/48} p^{23/24}   
		  +  (\|\a\|_2 \| \a\|_1)^{1/2} p^{3/4}  \right) 
		  \,,
	\end{equation}
	and if 
	$M^2 N^2 \|\FF{\a}\|^{12}_{L^{4/3}} < p \| \a\|_2^{12}$,  
	then  
\begin{equation}\label{f:Kloosterman_NM_2}
	S(\a,\beta) \lesssim \| \beta \|_2 \left(\|\FF{\a}\|^{6/7}_{L^{4/3}} \|\a \|^{1/7}_2 N^{1/7} M^{1/7} p^{13/14} 
	+  (\|\a\|_2 \| \a\|_1)^{1/2} p^{3/4} + p^{13/12} \| \FF{\a} \|_{L^{4/3}} \right) \,.
\end{equation}
\label{c:chi_Kloosterman}
\end{corollary}


Now  consider the case when our set $A$ is a collection of disjoint intervals. It is an important family of sets, including discrete fractal sets see, e.g., papers \cite{BD_AD_def}, \cite{chang2011partial}, 
\cite{MMS_popular}---\cite{ushanov2010larcher} and 
\cite{zaremba1972methode}. 

\begin{theorem}
    Let $\Lambda \subset \F_p$, $I=[N]$,  $A=I \dotplus \Lambda$, $|A|>p^{1-\epsilon}$,  and $\chi$ be a non--principal multiplicative character. Then there is an absolute constant $c_*>0$ such that 
\begin{equation}\label{f:chi_interval}
    |\sum_{x\in A\cap A^{-1}} \chi (x)|
    \le |A\cap A^{-1}| \cdot  N^{-c_*} 
    \ll \frac{|A|^2}{p} \cdot  N^{-c_*} \,,
\end{equation} 
    provided $N\ge p^{\epsilon/c_*}$.
\label{t:chi_interval}
\end{theorem}
\begin{proof} 
    We combine an appropriate version of Corollary \ref{c:chi_hyp} and the well--known Bourgain--Gamburd machine \cite{bourgain2008uniform} applied  to equation \eqref{eq:hyp} see, e.g.,  \cite{MMS_popular}. 
    Indeed, for any $x\in A\cap A^{-1}$, we have $x=i+\la$ such that 
    $1=(i+\la)(i'+\la')$, where $i,i'\in I$ and $\la,\la'\in \Lambda$. 
    Thus we in very deed arrive to equation 
    \eqref{eq:hyp}. Now $I(i) \le N^{-1} (I*\overline{I})(i)$, where $\overline{I} = [-N,N]$ and hence the number of solutions to the equation $1=(i+\la)(i'+\la')$ can be bounded above as $1=(j+a)(j'+a')$ with $a,a'\in A$ and $j,j'\in \overline{I}$ (times $N^{-2}$, of course). 
    In particular (see \cite{MMS_popular} or just Proposition \ref{p:chi_inc} and Corollary \ref{c:chi_hyp} above),  we get for an absolute constant $c \in (0,1]$ that 
\begin{equation}\label{tmp:AA*_upper}
    |A\cap A^{-1}| \le \frac{|A|^2 |\overline{I}|^2}{N^2 p} + O(N^{-2} |A| \cdot N^{2-c}) 
    \ll \frac{|A|^2}{p}
\end{equation}
    and hence the second estimate of \eqref{f:chi_interval} follows from the first one. Here we have used the conditions that $|A|>p^{1-\epsilon}$ and  $N \ge p^{\epsilon/c}$, which is satisfied if we put 
    $c_* = c/4$, say. 

    Similarly, let $h\in [N]$ be an integer parameter and write $I(i) = h^{-1} (H*I)(i)+\eps(i)$, where $H=[h]$ and $\| \eps\|_\infty =1$, $|\supp (\eps)| \le 2h$.
    In particular, we have $\| \eps\|^2_2 \le 2h$ and one can threat $\eps$ as a sum of two functions $\eps_1,\eps_2$ with supports on some  shifts  of the interval $H$. 
    Put $\tilde{\eps} = \eps_1 + \eps_2 : H \to [-1,1]$.
    As always let us write 
\[
    \sum_{x\in A\cap A^{-1}} \chi (x) = 
    \sum_{1=(i+\la)(i'+\la')} \chi(\la+i) \Lambda(\la) \Lambda(\la') I(i) I(i')
\]
\[
    =
    \sum_{1=(i+\la)(i'+\la')} \chi(\la+i) \Lambda(\la) \Lambda(\la') (h^{-1} (H*I)(i)+\eps(i)) (h^{-1} (H*I)(i')+\eps(i'))
\]
\[
    =
    h^{-2} \sum_{a,a',h,h' ~:~ (a+h)(a'+h')=1} A(a) A(a') H(h) H(h') \chi (a+h) + \mathcal{E}
    = \sigma + \mathcal{E} \,,
\]
    where the error term $\mathcal{E}$ can be estimated as (there are better bounds as the set $A$ is $I$--invariant and not just $H$--invariant) 
\[
    |\mathcal{E}| \le 
    2h^{-1} \sum_{a,a',h,h' ~:~ (a+h)(a'+h')=1} \Lambda(a) A(a') 
    |\tilde{\eps}(h)| H(h')
    +
    \sum_{a,a',h,h' ~:~ (a+h)(a'+h')=1} \Lambda(a) \Lambda(a') |\tilde{\eps}(h) \tilde{\eps}(h')|  
\]
\begin{equation}\label{tmp:29.04_2}
    \ll 
    \frac{|A|^2}{p} \left( \frac{h}{N} + \frac{h^2}{N^2} \right) + |A| \cdot  \left( \frac{h^{1-c}}{\sqrt{N}} + \frac{h^{2-c}}{N} \right) 
    \ll 
    \frac{|A|^2 h}{pN} + \frac{|A| h^{1-c}}{\sqrt{N}} \,.
\end{equation} 
    Here we have assumed that $h\le \sqrt{N}$ and applied the well--known Bourgain--Gamburd machine \cite{bourgain2008uniform}, \cite{MMS_popular}.
    Recall that this  result replaces Corollary \ref{c:chi_hyp} in the case when $X$, $Y$ 
    are intervals and $\chi \equiv  1$ (that is why we need two additional main terms in \eqref{tmp:29.04_2}). 
    It remains to estimate the sum $\sigma$ and to do this we can use  
the Bourgain--Gamburd machine one more time, namely, we 
apply our 
Corollary \ref{c:chi_hyp} and get $\sigma \ll |A| h^{-c}$. 
    Finally, combining the estimate for $\sigma$ and  bound \eqref{tmp:29.04_2} for the error term $\mathcal{E}$, choosing the parameter $h=[\sqrt{N}]$, we obtain  
\[
    \sum_{x\in A\cap A^{-1}} \chi (x)
    \ll 
    |A| h^{-c}
    + \frac{|A|^2 h}{pN} + \frac{|A| h^{1-c}}{\sqrt{N}} 
    \ll 
    |A| h^{-c} \ll |A| N^{-c/2} \ll \frac{|A|^2}{p} \cdot N^{-c/4}
\]
    thanks to our assumptions $|A|>p^{1-\epsilon}$ and  $N \ge p^{4\epsilon/c}$.
    The same calculations show that there is an asymptotic formula for $|A\cap A^{-1}|$  and, in particular, the inverse inequality to \eqref{tmp:AA*_upper} takes place. It gives us the first inequality in \eqref{f:chi_interval}.
    This completes the proof. 
$\hfill\Box$
\end{proof}

\bigskip

It is well--known and it is easy to see that the multiplicative equation \eqref{eq:hyp} is 
almost coincides 
(up to some transformation) with  the additive equation  
\[
 \frac{1}{x+a} - \frac{1}{y+b} \equiv 1 \pmod p \,,
\]
where $a\in A$, $b\in B$, $x\in X$, $y\in Y$. 
Thus we obtain an analogue of Theorem \ref{t:chi_interval}. 

\begin{theorem}
    Let $\Lambda \subset \F_p$, $I=[N]$, $A=I \dotplus \Lambda$, $|A|>p^{1-\epsilon}$,  and $\chi$ be a non--principal multiplicative character. Then there is an absolute constant $c_*>0$ such that 
\begin{equation}\label{f:chi_interval+}
    |\sum_{x\in A^{-1} \cap (A^{-1}+1)} \chi (x)|
    \le |A^{-1} \cap (A^{-1}+1)| \cdot  N^{-c_*} 
    \ll \frac{|A|^2}{p} \cdot  N^{-c_*} \,,
\end{equation} 
    provided $N\ge p^{\epsilon/c_*}$.
\label{t:chi_interval+}
\end{theorem}


Now we are ready to prove Theorem \ref{t:ZM_new} from the introduction. 

\bp 

\begin{proof}
    We follow the scheme and the notation of the proof 
    from paper 
    \cite[Pages 3--7]{MMS_Korobov}. 
    It was shown that the set of $a\in A$ satisfying \eqref{f:ZM_expansion} contains a set of the form 
    $Z_M \cap Z^{-1}_M$, $|A| \sim |Z_M|^2/p$, 
    $|Z_M|\sim p^{w_M+2\eps(1-w_M)}$ and the set $Z_M$ is a disjoint union of some shifts of an interval of length $N\sim p^{2\eps}$, where $\eps\gg 1/M$ is a parameter and Hausdorff dimension $w_M$ enjoys the asymptotic formula $w_M = 1-O(1/M)$, $M\to \infty$. 
    Thus we can apply Theorem \ref{t:chi_interval} and write 
\[
    |A\cap \G| = (p-1)^{-1} \sum_{\chi} \left(\sum_{x\in A} \chi(x) \right)  \left(\sum_{x\in \G} \ov{\chi}(x) \right) \ge  \frac{|A||\G|}{p-1} - C|A| N^{-c_*} >0 \,,
\]
    where $C,c_*>0$ are some absolute constants. 
    Here we have used conditions \eqref{c:Gamma}, \eqref{c:M}, the fact that $M \sim \frac{\log p}{\log \log p}$ and $\eps \gg 1/M$. 
    It remains to check that $N\ge p^{\epsilon/c_*}$ or, in other words, that $\eps \gg \epsilon$. 
    Since $|Z_M|\sim p^{w_M+2\eps(1-w_M)} = p^{1-\epsilon}$, it follows that $\epsilon =(1-w_M)(1-2\eps)\ll 1/M$ and 
    thus the required condition takes place. 
    This completes the proof. 
$\hfill\Box$
\end{proof}

\bp 

    Let us make a final remark. 
    Loosely, 
    Theorem \ref{t:ZM_new} gives us a non--trivial bound for the {\it multiplicative energy} of the set $Z_M$, see formula \eqref{def_tmp:mult_energy} below. 
    Nevertheless, the last fact follows from the circumstance that $Z_M$ is an Ahlfors--David set, \cite{BD_AD_def}, that is 
    for an arbitrary $z\in Z_M$ one has 
\begin{equation}\label{def:AD}
    |Z_M \cap (\mathcal{D}+z)| \sim_M  |\mathcal{D}|^{w_M} N^{1-w_M} 
\end{equation}
    for any interval $\mathcal{D}$, $|\mathcal{D}| \ge N$ with the center at the origin. 
    Recall that 
    in  \cite{BD_AD_def} a non--trivial upper bound was obtained  for the {\it additive energy} of any Ahlfors--David set. 
    Let us briefly 
    prove an upper estimate for the multiplicative energy of 
    an arbitrary 
    Ahlfors--David set $Z_M$, having large Hausdorff dimension $w_M$.
    The advantage of bound \eqref{f:E^*(Z)} below that our power saving can be expressed in terms of $|Z_M|$ but not just $N$.

    Namely, write  $Z=Z_M$, $w=w_M$ and then $|Z| \sim_M p^w N^{1-w}$. 
    Also, put $\delta \sim_M \D^{w} N^{1-w}$, where  
    $\D$ is a parameter.
    By the points/planes incidences in $\F_p$ (see \cite{Rudnev_points/planes}) and property \eqref{def:AD}  one has 
\begin{equation}\label{def_tmp:mult_energy}
    \E^\times (Z) := |\{ (z_1,z_2,z_3,z_4) \in Z^4 ~:~ z_1 z_2 = z_3 z_4 \}| 
\end{equation}
\[
    \ll \d^{-2} |\{ (z_1,z_2,z'_1,z'_2,d,d') \in Z^4 \times [ \D ]^2 ~:~ z_1 (z_2+d) \equiv z'_1 (z'_2+d') \}|
\]
\[
    \ll 
    \d^{-2}
    \left( \frac{|Z|^4 \D^2}{p} + |Z|^3 \D^{3/2}  \right) 
    \ll \d^{-2} p^3  \,,
\]
    where the optimal choice for $\D$ is $\D = (p/|Z|)^2$. 
    Thus 
\begin{equation}\label{f:E^*(Z)}
    \E^\times (Z) \ll_M |Z|^3 (p/|Z|)^{3-4w} N^{-2 (1-w)} 
    \sim_M |Z|^3 \cdot |Z|^{-\frac{(4w-3)(1-w)}{w}} N^{\frac{-(1-w)(3-2w)}{w}} 
    < |Z|^3 
\end{equation}
    for $w >3/4$. 
    Thus, we have a power saving in terms of $|Z|$ for the multiplicative energy of any Ahlfors--David set. 

\section{Data availability and conflicts of interest}

No data was used for the research described in the article.
Ilya D. Shkredov declares no conflicts of interest. 

\bibliographystyle{abbrv}

\bibliography{bibliography}{}

\end{document}